\documentclass[preprint,11pt]{elsarticle}

\usepackage{amsfonts, amsmath, amscd}
\usepackage[psamsfonts]{amssymb}

\usepackage{amssymb}

\usepackage{pb-diagram}

\usepackage[all,cmtip]{xy}

\usepackage{amsthm}

\usepackage[usenames]{color}

\newtheorem{theorem}{Theorem}[section]

\newtheorem{corollary}[theorem]{Corollary}

\newtheorem{proposition}[theorem]{Proposition}

\newtheorem{problem}[theorem]{Problem}

\theoremstyle{definition}
\newtheorem{example}[theorem]{Example}

\theoremstyle{remark}

\numberwithin{equation}{section}
\numberwithin{theorem}{section}

\newcommand{\w}{\omega}

\newcommand{\ZZ}{\mathbb{Z}}

\newcommand{\IR}{\mathbb{R}}

\newcommand{\AAA}{\mathcal{A}}

\newcommand{\Ra}{\Rightarrow}

\newcommand{\SM}{{\setminus}}

\begin{document}

\begin{frontmatter}

\title{On $\kappa$-Fr\'{e}chet--Urysohn topological groups}

\author{Saak Gabriyelyan}
\ead{saak@math.bgu.ac.il}
\address{Department of Mathematics, Ben-Gurion University of the Negev, Beer-Sheva, Israel}

\author{Alexander V. Osipov}
\ead{OAB@list.ru}
\address{Krasovskii Institute of Mathematics and Mechanics, Ural Federal University, Yekaterinburg, Russia}

\author{Evgenii Reznichenko}
\ead{erezn@inbox.ru}
\address{Department of Mathematics, Lomonosov Mosow State University, Moscow, Russia}

\begin{abstract}
We characterize $\kappa$-Fr\'{e}chet--Urysohn topological groups. Using this characterization we show that: (1) a hemicompact topological group is $\kappa$-Fr\'{e}chet--Urysohn iff it is locally compact, and (2) if $F$ is a closed metrizable subspace of a topological vector space (tvs) $E$ such that the quotient $E/F$ is a $\kappa$-Fr\'{e}chet--Urysohn space, then also $E$ is a $\kappa$-Fr\'{e}chet--Urysohn space. Consequently, the product of a $\kappa$-Fr\'{e}chet--Urysohn tvs and a metrizable tvs is a $\kappa$-Fr\'{e}chet--Urysohn space. Under Martin's Axiom, we construct a countable Boolean $\kappa$-Fr\'{e}chet--Urysohn group which is not a $k_\IR$-space.
\end{abstract}


\end{frontmatter}


\section{Introduction}


The class of $\kappa$-Fr\'{e}chet--Urysohn spaces was introduced by Arhangel'skii. A Tychonoff space $X$ is called {\em $\kappa$-Fr\'{e}chet--Urysohn} if, for any open set $U\subseteq X$ and each point $x\in \overline{U}$, there exists a sequence $\{x_n\}_{n\in \w}\subseteq U$ that converges to~$x$. It should be mentioned that without using the term $\kappa$-Fr\'{e}chet--Urysohn, Mr\'{o}wka proved in \cite{Mrowka} that any product of first countable spaces is  $\kappa$-Fr\'{e}chet--Urysohn. $\kappa$-Fr\'{e}chet--Urysohn spaces are thoroughly studied by Liu and Ludwig \cite{LiL}, Sakai \cite{Sak2,Sakai}, Gabriyelyan \cite{Gabr-B1,Gabr-seq-Ascoli}, Gabriyelyan and Reznichenko \cite{GR-compact}, and Gabriyelyan, Osipov and Reznichenko \cite{GOR-Baire}.

In this article we continue the study of topological groups (in particular, topological vector spaces) which are $\kappa$-Fr\'{e}chet--Urysohn spaces. In Theorem \ref{t:kFU-group-charac} we characterize $\kappa$-Fr\'{e}chet--Urysohn topological groups. Using this characterization we show in Proposition \ref{p:kFU-hemicompact} that a hemicompact topological group $G$ is $\kappa$-Fr\'{e}chet--Urysohn if, and only if, $G$ is Baire if, and only if, $G$ is locally compact. It is known that the product of a Fr\'{e}chet--Urysohn space and a non-discrete metrizable space can be not $\kappa$-Fr\'{e}chet--Urysohn, see \cite{GR-compact}. However, in the realm of topological vector spaces the situation changes. Indeed, we show in Corollary \ref{c:prod-kFU-metr-tvs} that if $E$ is a $\kappa$-Fr\'{e}chet--Urysohn topological vector space and $F$ is a metrizable topological vector space, then the product $E\times F$ is a $\kappa$-Fr\'{e}chet--Urysohn space. Under Martin's Axiom, we construct a countable Boolean $\kappa$-Fr\'{e}chet--Urysohn group which is not a $k_\IR$-space, see Example \ref{exa:kFU-count-non-kR}.


\section{Main results} \label{sec:kFU-charac}



We start with the following characterization of topological groups which are $\kappa$-Fr\'{e}chet--Urysohn spaces, we shall call such groups by {\em  $\kappa$-Fr\'{e}chet--Urysohn groups}. The identity of a group $G$ is denoted by $e$.

Following  \cite{Osipov-23}, a family $\AAA=\{A_n\}_{n\in\w}$ of subsets of a space $X$ {\em weakly converges} to a point $x$ if for every neighborhood $W$ of $x$, there are a sequence $\{a_n\}_{n\in\w}$ and $m\in\w$ such that $a_n\in A_n$ $(n\in\w)$ and $a_n\in W$ for every $n>m$.


\begin{theorem} \label{t:kFU-group-charac}
For a topological group $G$, the following assertions are equivalent:
\begin{enumerate}
\item[{\rm(i)}] for every sequence $\{\mathcal{V}_n: n\in \w\}$ of countable families $\mathcal{V}_n=\{V^i_n: i\in \w\}$ of open subsets of $G$ such that $\mathcal{V}_n$ weakly converges to $e$ for every $n\in\w$, there are strictly increasing sequences  $(i_k)$ and $(n_k)$ in $\w$ and a sequence $\{x_k\}_{k\in\w}$ such that $x_k\to e$ and $x_k\in V^{i_k}_{n_k}$ for every $k\in\w$;
\item[{\rm(ii)}] for every sequence $\{U_n\}_{n\in\w}$ of open subsets of $G$ such that $e\in \overline{U_n}$ for every $n\in\w$, there are a strictly increasing sequence $(n_k)\subseteq \w$ and a sequence $\{x_k\}_{k\in\w}$ in $G$ such that $x_k\in U_{n_k}$ $(k\in\w)$ and $x_k\to e$;
\item[{\rm(iii)}] $G$ is a $\kappa$-Fr\'{e}chet--Urysohn space.
\end{enumerate}
\end{theorem}

\begin{proof}
(i)$\Ra$(ii) For every $i\in \w$, we set $V^i_n:=U_n$. Now (ii) immediately follows from (i).
\smallskip

(ii)$\Ra$(iii) Let $U$ be an open subset of $G$ such that $e\in \overline{U}\SM U$. For every $n\in\w$, set $U_n:=U$. Then (iii) immediately follows from (ii).
\smallskip

(iii)$\Ra$(i) Let $\{\mathcal{V}_n: n\in \w\}$  be a family satisfying (i). Choose an arbitrary one-to-one sequence $\{p_n\}_{n\in\w}$ in $G\SM \{e\}$  converging to $e$ (such a sequence exists by the $\kappa$-Fr\'{e}chet--Urysohness applying to the open set $G\SM\{e\}$). For every $n\in\w$, set $p_n\cdot\mathcal{V}_n:=\{p_n V^i_n: i\in \w\}$. Choose a sequence $\{W_n: n\in \w\}$ of open neighborhoods of $e$ such that
\begin{equation} \label{equ:kFU-group-charac-1}
e\not\in \overline{p_n W_n}
\end{equation}
for every $n\in \w$. For each $0\leq n\leq i$, set
\begin{equation} \label{equ:kFU-group-charac-2}
\mathcal{O}^i_n=p_nV^i_n\cap p_nW_n.
\end{equation}
We claim that $e\in \overline{\bigcup \{\mathcal{O}^i_n: 0\leq n\leq i\}}$.  Indeed, let $U$ be a neighborhood of $e$. Choose a neighborhood $U_1$ of $e$ such that $U_1\cdot U_1\subseteq U$. Since $p_n\to e$, there is $m\in\w$ such that $p_n\in U_1$ for each $n\geq m$.  As $\mathcal{V}_m$ weakly converges to $e$, there are $j\geq m $ and a point $s_m^j\in V_m^j$ such that $s_m^j\in U_1\cap W_m$. Then $p_m\cdot s_m^j\in U_1\cdot U_1\subseteq U$ and $p_m\cdot s_m^j\in \mathcal{O}^j_m$.
\smallskip

Since $G$ is $\kappa$-Fr\'{e}chet--Urysohn, there is a sequence $(z_k)\subseteq \bigcup \{\mathcal{O}^i_n: 0\leq n\leq i\}$ converging to $e$. For every $k\in\w$, choose indices  $i_k$ and $n_k$ such that $z_k\in \mathcal{O}^{i_k}_{n_k}$. The sequence $(n_k)$ cannot be bounded because of (\ref{equ:kFU-group-charac-1}) and the inclusion $\mathcal{O}^{i_k}_{n_k} \subseteq p_{n_k}W_{n_k}$. Taking into account that $j_k\geq n_k$ it follows that also $(j_k)$ is unbounded. Passing to subsequences if needed, we can assume that $(j_k)$ and $(n_k)$ are strictly increasing.  For every $k\in\w$, set $x_k:= p_{n_k}^{-1}\cdot z_k $. Then, by (\ref{equ:kFU-group-charac-2}), $x_k\in V^{i_k}_{n_k}$ and, clearly, $x_k\to e$. Thus (i) is satisfied.
\end{proof}


As an application of Theorem \ref{t:kFU-group-charac} we obtain a structure of hemicompact groups with the $\kappa$-Fr\'{e}chet--Urysohn property.

\begin{proposition} \label{p:kFU-hemicompact}
For a hemicompact topological group  $G$ the following assertions are equivalent:
\begin{enumerate}
\item[{\rm(i)}] $G$ is $\kappa$-Fr\'{e}chet--Urysohn;
\item[{\rm(ii)}]  $G$ is locally compact;
\item[{\rm(iii)}] $G$ is Baire.
\end{enumerate}
\end{proposition}

\begin{proof}
(i)$\Ra$(ii) Let $G$ be $\kappa$-Fr\'{e}chet--Urysohn, and suppose for a contradiction that $G$ is not locally compact. Take an increasing sequence $\{K_n\}_{n\in\w}$ of compact subsets of $G$ such that any compact subset of $G$ is contained in some $K_n$. We assume that $e\in K_0$. By our supposition all $K_n$ are not neighborhoods of $e$. For every $n\in\w$, set $U_n:=G\SM K_n$. Then $U_n$ is an open subset of $G$ such that $e\in \overline{U_n}\SM U_n$. Since  $G$ is $\kappa$-Fr\'{e}chet--Urysohn, (ii) of Theorem \ref{t:kFU-group-charac} implies that there are a strictly increasing sequence $(n_k)\subseteq \w$ and a sequence $\{x_k\}_{k\in\w}$ in $G$ such that $x_k\in U_{n_k}$ $(k\in\w)$ and $x_k\to e$. Since $K:=\{x_k\}_{k\in\w}\cup\{e\}$ is a compact subset of $G$, there is $m\in \w$ such that $K\subseteq K_m$. Take $k\in\w$ such that $n_k>m$. Then $x_k\in K_m\cap  U_{n_k}=K_m\cap (G\SM  K_{n_k})=\emptyset$. This is a contradiction.

(ii)$\Ra$(i) If $G$ is locally compact, then $G$ is $\kappa$-Fr\'{e}chet--Urysohn by Corollary 6.6 
of \cite{GR-compact}.

(ii)$\Ra$(iii) is trivial.

(iii)$\Ra$(ii) If $G$ is Baire, then (see (i)), there is $n\in\w$ such that the compact set $K_n$ contains an open subset. Thus $G$ is locally compact.
\end{proof}
The condition of being $\kappa$-Fr\'{e}chet--Urysohn in Proposition \ref{p:kFU-hemicompact} cannot be replaced by the property of being a sequential space. Indeed, it is well-known that the free abelian group $A([0,\w])$ over the convergent sequence $[0,\w]$ is a sequential hemicompact group which is not locally compact.

Example 14.3 of \cite{kak} states that an abelian hemicompact topological group $G$ is Fr\'{e}chet--Urysohn if, and only if, it is a locally compact Polish group. Recall that a Tychonoff space $X$ is {\em Fr\'{e}chet--Urysohn} if, for any $A\subseteq X$ and each $x\in \overline{A}$, there exists a sequence $\{a_n: n\in \w\}\subseteq A$ that converges to $x$. Below we show that the condition of being abelian can be omitted.

\begin{corollary} \label{c:FU-hemicompact}
Let $G$ be a hemicompact topological group. Then $G$ is Fr\'{e}chet--Urysohn if, and only if, it is a locally compact Polish group.
\end{corollary}

\begin{proof}
Assume that $G$ is Fr\'{e}chet--Urysohn. Then, by Proposition \ref{p:kFU-hemicompact}, $G$ is locally compact. The group $G$ being Fr\'{e}chet--Urysohn has countable tightness. Recall that, by \cite{Ismail}, the character of any locally compact group coincides with its tightness. Therefore $G$ has countable character, and hence, by the Kakutani metrization theorem, $G$ is metrizable. As $G$ is hemicompact and metrizable, it is separable. Thus $G$ is  a locally compact Polish group. The converse assertion is clear.
\end{proof}

Proposition 14.4 of \cite{kak} states that if $F$ is a closed metrizable subspace of a tvs $E$ such that $E/F$ is Fr\'{e}chet--Urysohn, then the space $E$ is a Fr\'{e}chet--Urysohn space, too. Below we show that an analogous result holds also for $\kappa$-Fr\'{e}chet--Urysohn spaces.

\begin{proposition} \label{p:kFU-metriz-quot}
If $F$ is a closed metrizable subspace of a topological vector space $E$ such that the quotient $E/F$ is a $\kappa$-Fr\'{e}chet--Urysohn space, then  $E$ is also a $\kappa$-Fr\'{e}chet--Urysohn space.
\end{proposition}

\begin{proof}
Since $F$ is a metrizable subspace of $E$, there exists a decreasing sequence $\{V_n\}_{n\in\w}$ of balanced neighborhoods of zero in $E$ such that $V_{n+1}+V_{n+1}\subseteq V_n$ for all $n\in \w$ and the sequence $\{V_n\cap F\}_{n\in\w}$ is a basis of neighborhoods of zero in $F$. To prove that $E$ is
$\kappa$-Fr\'{e}chet--Urysohn, it is sufficient to show that if $U\subseteq E$ is an open set and $0\in \overline{U}\SM U$, then there exists a sequence in $U$ converging to $0$.

Let $p: E\rightarrow E/F$ be the quotient map. 
Note that $0\in \overline{V_n\cap U}$  for all $n\in \omega$. Hence $p(0)\in \overline{p(V_n\cap U)}$. Since $V_n\cap U$ is open and $p$ is an open map, the set $p(V_n\cap U)$ is open in $E/F$. Hence, by Theorem \ref{t:kFU-group-charac}(ii),  there are a sequence $(n_k)\subseteq \w$ and  a sequence $\{x_{k}\}_{k\in\w}\subseteq E$ such that
\[
p(x_{k})\in p(V_{n_k}\cap U) \; \mbox{ and } \; p(x_{k})\rightarrow p(0) \mbox{ as } k\to \infty.
\]
In particular, for every $k\in\w$, there are $z_k\in V_{n_k}\cap U$ and  $f_k\in F$ such that $x_k=z_k+f_k$. Replacing $x_k$ by $z_k$ if needed, we can assume that $x_k\in V_{n_k}\cap U$ for every $k\in\w$.

Fix a balanced neighborhood $W$ of zero in $E$. Choose $n_0\in \w$ such that $(V_{n_0}+V_{n_0})\cap F\subseteq W$. Taking into account that $p(W\cap V_{n_0})$ is an open neighborhood of zero in $E/F$, $W\cap V_{n_0}$ is balanced and $p(x_{k})\to 0$, we obtain that there is an $m>n_0$ such that
\[
\big[x_k+(W\cap V_{n_0})\big]\cap F\neq \emptyset \quad \mbox{ for all $k>m$.}
\]
Therefore, for every $k>m$, there exist $y_k\in W\cap V_{n_0}$ and  $u_k\in F$, such that $x_k+y_k=u_k$. Since $x_k\in V_{n_k}\subseteq V_k\subseteq V_{n_0}$ and $y_k\in V_{n_0}$, we deduce that $u_k=x_k+y_k\in (V_{n_0}+V_{n_0})\cap F\subseteq W$. Therefore $x_k=u_k-y_k\in W+W$. Since $W$ was arbitrary, this implies that $\{x_k\}_{k\in \w}$ is a sequence in $U$ that converges to $0$ in $E$. Thus $E$ is a $\kappa$-Fr\'{e}chet--Urysohn space.
\end{proof}

Michael showed (see \cite[Corollary~14.3]{kak}) that the product of a Fr\'{e}chet--Urysohn tvs and a metrizable tvs is Fr\'{e}chet--Urysohn. Proposition \ref{p:kFU-metriz-quot} shows that an analogous result holds true also for $\kappa$-Fr\'{e}chet--Urysohn tvs.

\begin{corollary} \label{c:prod-kFU-metr-tvs}
If $E$ is a $\kappa$-Fr\'{e}chet--Urysohn topological vector space and $F$ is a metrizable topological vector space, then the product $E\times F$ is a $\kappa$-Fr\'{e}chet--Urysohn space.
\end{corollary}
Note that in Corollary \ref{c:prod-kFU-metr-tvs}, the condition that $E$ and $F$ are topological vector spaces is essential. Indeed, in Theorem 3.7 of \cite{GR-compact} it is shown that the product of the Fr\'{e}chet--Urysohn fun $V(\w)$ with an arbitrary non-discrete space $X$ is not $\kappa$-Fr\'{e}chet--Urysohn.

Being motivated by the famous Malykhin problem (which asks whether there exists a countable non-metrizable Fr\'{e}chet--Urysohn group), the following problem arises naturally (note that this problem complements Problem 2.1 of \cite{GR-compact}).
\begin{problem} \label{prob:kFU-not-kR}
Does there exist a countable  $\kappa$-Fr\'{e}chet--Urysohn group which is not metrizable?
\end{problem}
Under Martin's Axiom, in Example \ref{exa:kFU-count-non-kR} below we answer Problem \ref{prob:kFU-not-kR} in the affirmative in a stronger form.

For a (zero-dimensional) Tychonoff space  $X$, we denote by $C_p(X)$ (resp., $C_p(X,\mathbf{2})$) the space $C(X)$ (resp., $C(X,\mathbf{2})$) of all continuous functions from $X$ to $\IR$ (resp., from $X$ to the doubleton $\mathbf{2}=\ZZ(2)$ considered as a discrete two element group) endowed with the pointwise topology. The space $B_1(X,\mathbf{2})$ is the space of all functions $f:X\to \mathbf{2}$ which are pointwise limits of sequences from $C_p(X,\mathbf{2})$, it is also endowed with the pointwise topology.

The characteristic function of a subset $A$ of a set $\Omega$ is denoted by $\mathbf{1}_A$.

Recall that a separable metrizable space $X$ is called a {\em $Q$-set} if every subset of $X$  is of type $F_\sigma$ in $X$. We shall use the next assertion.

\begin{proposition} \label{p:Q-set=Q2-set}
A separable zero-dimensional metrizable space $X$ is  a $Q$-set if, and only if, $B_1(X,\mathbf{2})=\mathbf{2}^X$.
\end{proposition}

\begin{proof}
Assume that $X$ is a $Q$-set. Let $f\in \mathbf{2}^X$ be arbitrary. Set  $M:=\{x\in X: f(x)=1\}$. Then $f=\mathbf{1}_M$. Since $X$ is a $Q$-set, we have $M=\bigcup_n A_n$ and $X\SM M=\bigcup_{n\in\w} B_n$ for some increasing sequences $\{A_n\}_{n\in\w}$ and  $\{B_n\}_{n\in\w}$ of closed subsets of $X$, respectively. As $X$ is zero-dimensional and Lindel\"{o}f, Theorem 6.2.7 of \cite{Eng} implies that $X$ is strongly zero-dimensional. Therefore, for every $n\in\w$, there is  a clopen subset $C_n$ of $X$ such that $A_n \subseteq C_n \subseteq X\SM B_n$. Then $\mathbf{1}_{C_n}\to \mathbf{1}_M=f$. Thus  $B_1(X,\mathbf{2})=\mathbf{2}^X$.
\smallskip

Conversely, assume that  $B_1(X,\mathbf{2})=\mathbf{2}^X$. To show that $X$ is a $Q$-set, let $M$ be an arbitrary subset of $X$. Then $\mathbf{1}_M\in B_1(X,\mathbf{2})$. Let $\{ \mathbf{1}_{A_n}\}_{n\in\w}$ be  a sequence in $C_p(X,\mathbf{2})$ which pointwise  converges to $\mathbf{1}_M$. It is clear that all $A_n$ are clopen subsets of $X$ and
\[
M=\bigcup_{k\in\w} \bigcap_{n\geq k} A_n.
\]
Since, for every $k\in\w$, the intersection $ \bigcap_{n\geq k} A_n$ is closed, we obtain that $M$ is an $F_\sigma$-set in $X$. Thus $X$ is a $Q$-set.
\end{proof}

A Tychonoff space $X$ is a {\em $k_\IR$-space} if every $k$-continuous function $f:X\to\IR$ is continuous. Recall that a function $f:X\to\IR$ is called {\em $k$-continuous} if the restriction of $f$ onto each compact subset of $X$ is continuous.
Recall also that a separable metrizable space $X$ is called a {\em $\gamma$-set} if the space $C_p(X)$ is Fr\'{e}chet--Urysohn.

\begin{example} \label{exa:kFU-count-non-kR}
Under $\mathrm{MA}$, there is a countable Boolean  $\kappa$-Fr\'{e}chet--Urysohn group $G$ which is not a $k_\IR$-space.
\end{example}

\begin{proof}
First we prove  the following claim.
\smallskip

{\em Claim 1. If there is a $\gamma$-set $X$ which is not a $Q$-set, then there is a countable Boolean  $\kappa$-Fr\'{e}chet--Urysohn group $G$ which is not a $k_\IR$-space.}
\smallskip

{\em Proof of Claim 1.} It is well-known that any $\gamma$-set is zero-dimensional, see \cite[Corollary]{GN}. It follows that the group $C_p(X,\mathbf{2})$ is separable and Hausdorff. Let $H$ be a countable and dense  subgroup of $C_p(X,\mathbf{2})$. Since $X$ is a $\gamma$-set, the space $C_p(X)$ and hence its subgroup $C_p(X,\mathbf{2})$ and the group $H$ are Fr\'{e}chet--Urysohn. As $X$  is not a $Q$-set, Proposition \ref{p:Q-set=Q2-set} implies that there is a function $g\in\mathbf{2}^X\SM B_1(X,\mathbf{2})$. Set
\[
G:=H\cup(g+H) \subseteq \mathbf{2}^X.
\]
It is clear that $G$ is a countable subgroup of $\mathbf{2}^X$. We show that $G$ is as desired.

Observe that the group $H$ is dense in $\mathbf{2}^X$ and, hence, $H$ and $g+H$ are dense in $G$. By construction, $H$ and hence also $g+H$ are Fr\'{e}chet--Urysohn spaces. Therefore, by \cite[Theorem~2.1]{Gabr-B1}, the group $G$ is a  $\kappa$-Fr\'{e}chet--Urysohn space.

It remains to show that $G$ is not a $k_\IR$-space. To this end, it suffices to show that the characteristic function $\mathbf{1}_H$ of $H$, which is discontinuous by the density of $H$ in $G$,  is $k$-continuous. To prove that $\mathbf{1}_H$ is $k$-continuous we show that for every compact subset $K$ of $G$, the sets $L:=K\cap H$ and $M:=K\cap (g+H)$ are compact. By the symmetry of $L$ and $M$, we show that $L$ is a closed subset of $G$.

Assume that $h\in \overline{L}$, and suppose for a contradiction that $h\not\in H$ and hence $h\in g+ H$.
Since $\overline{L}$ is a closed subset of the compact countable set $K$,  $\overline{L}$ is metrizable. Therefore, there is a sequence $\{g_n\}_{n\in\w}$ in $L$ converging to $h$. Take $h_0\in H$ such that $h=g+h_0$. Then $g_n-h_0\to g$, and hence $g\in B_1(X,\mathbf{2})$. But  this contradicts the choice of the function $g$ that finishes the proof of Claim 1.
\smallskip

To finish the proof of the example, we note that Theorem 1 of \cite{Gal-Miller} states that, under Martin's Axiom, there is a $\gamma$-set $X\subseteq \IR$ of cardinality the continuum. On the other hand, any $Q$-set has cardinality strictly less than $\mathfrak{c}$, see the discussion after Theorem 8.54 of \cite[p.~314]{Bukovsky2011}. Therefore the $\gamma$-set $X$ is not a $Q$-set. Now Claim 1 applies.
\end{proof}


\bibliographystyle{amsplain}

\end{document}